\documentclass[11pt]{article}

\usepackage{amsmath}
\usepackage{amssymb}
\usepackage{amsthm}
\usepackage{hyperref}

{\theoremstyle{plain}%
 \newtheorem{theorem}{Theorem}

}
{\theoremstyle{remark}

}
{\theoremstyle{definition}

}

\allowdisplaybreaks

\begin{document}

 \title{On a $_2F_1\big(\frac{1}{4}\big)$-identity due to Gosper}
\date{April 2026}
\author{Cetin Hakimoglu-Brown\thanks{Email address: {\tt mathemails@proton.me}}\\ Berkeley, CA, USA}

\maketitle 

\begin{abstract}
 It is only in exceptional cases that a $_2F_1(z)$-series with rational parameters and a rational argument, apart from the cases for 
 $z \in \{ \pm 1, \frac{1}{2} \}$ associated with classical hypergeometric identities, admits an evaluation given by a combination of 
 $\Gamma$-values with rational arguments. In this paper, we present a new and integration-based approach toward the construction 
 of special values for $_2F_1$-series of the desired form. We apply this approach using a $_2F_1\big(\frac{1}{4}\big)$-identity 
 originally due to Gosper and later considered by Vidunas, Ebisu, and Zudilin, to evaluate a ${}_{2}F_{1}$-series of convergence rate
 $\big(\frac{172872}{185039}\big)^2$. With regard to extant research on so-called ``strange'' ${}_{2}F_{1}$-evaluations, as in the work 
 of Ebisu and Zeilberger, our new series seems to have the largest numerator/denominator in its argument. 
\end{abstract}

\section{Introduction}
 The rising factorial or Pochhammer symbol is defined so that $(x)_0 = 1$ and so that $(x)_n = x (x+1) \cdots (x + n - 1)$ for a positive 
 integer $n$. We make use of the notational shorthand such that 
\begin{equation*}
 \left[ \begin{matrix} \alpha, \beta, \ldots, \gamma \vspace{1mm} \\ 
 A, B, \ldots, C \end{matrix} \right]_{n} = \frac{ (\alpha)_{n} (\beta)_{n} 
 \cdots (\gamma)_{n} }{ (A)_{n} (B)_{n} \cdots (C)_{n}}. 
\end{equation*}
 Generalized hypergeometric series may be defined so that 
\begin{equation}\label{displaypFq}
 {}_{p}F_{q}\!\!\left[ \begin{matrix} 
 a_{1}, a_{2}, \ldots, a_{p} \vspace{1mm}\\ 
 b_{1}, b_{2}, \ldots, b_{q} \end{matrix} \ \Bigg| \ x 
 \right] = \sum_{n=0}^{\infty} 
 \left[ \begin{matrix} 
 a_{1}, a_{2}, \ldots, a_{p} \vspace{1mm} \\ 
 b_{1}, b_{2}, \ldots, b_{q} \end{matrix} \right]_{n} \frac{z^n}{n!}, 
\end{equation}
 with reference to standard texts on such series, as in the monograph from Slater \cite{Slater1966}. Since $_pF_q$-series are of central 
 importance in the application of special functions, this serves as a basis for the development of new methods to express series as in 
 \eqref{displaypFq}, for specified values for $z$ (the \emph{argument} or \emph{convergence rate}) and for  expressions of the
 forms $a_i$ and $b_i$ (the \emph{parameters}). Typically, such expressions for $_pF_q$ series would be given as combinations of values 
 of the $\Gamma$-function defined via the Euler integral 
\begin{equation}\label{displayGamma}
 \Gamma(x) = \int_0^{\infty} t^{x-1} e^{-t} \, dt
\end{equation} 
 for $\Re(x) > 0$, referring to the classic text by Rainville \cite[\S2]{Rainville1960} for background material on special functions as in 
 \eqref{displayGamma}. \emph{Gaussian hypergeometric functions} refer to functions given by $_2F_1(z)$-series and may be seen as 
 providing the most fundamental instances of generalized hypergeometric series. The foregoing points naturally lead toward research 
 topics given by the expression of $_2F_1$-series with rational parameters and arguments in terms of the $\Gamma$-function, apart 
 from the cases given by combinations of arguments/parameters corresponding to 
 classical hypergeometric identities. Such research topics have been 
 explored, in remarkable ways, by Joyce and Zucker \cite{JoyceZucker1991,ZuckerJoyce2001,JoyceZucker2002}, Ebisu \cite{Ebisu2017}, 
 and Zudilin \cite{Zudilin2025}. In this paper, we introduce a technique, relying on the use of cyclotomic polynomials obtained via the 
 manipulation of an integral identity from Euler, for constructing new evaluations for $_2F_1$-series in the vein of the Joyce--Zucker 
 $_2F_1$-evaluations. We apply this technique using a $_2F_1\big(\frac{1}{4}\big)$-series identity originally due to 
 Gosper\footnote{Given in an unpublished letter to D.\ Stanton, XEROX Palo Alto Research Center, 21 December 1977.} and later 
 considered by Vidunas \cite{Vidunas2002}, Ebisu\footnote{See also \url{https://arxiv.org/abs/1607.04742}.} \cite{Ebisu2017}, and 
 Zudilin \cite{Zudilin2025}. 

 Informally, the key to our technique relies on the use of 
 known $_2F_1$-identities together with manipulations of the classical relation, dating back to Euler, whereby 
\begin{equation}\label{Eulerint}
 {}_{2}F_{1}\!\!\left[ \begin{matrix} 
 a, b \vspace{1mm}\\ 
 c \end{matrix} \ \Bigg| \ z 
 \right] = 
\frac{\Gamma(c)}{\Gamma(b)\Gamma(c-b)}
\int_0^1
t^{b-1}(1-t)^{c-b-1}(1-zt)^{-a}\,dt 
\end{equation}
 for $|z| < 1$ and $\Re(c) > \Re(b) > 0$ \cite[\S4.30]{Rainville1960} \cite[\S1.1]{Slater1966}, in such a way so as to produce, in a specific 
 way, a cyclotomic polynomial within the integrand of a variant of the right-hand side of \eqref{Eulerint}. 

 The most fundamental out of the classically known ${}_{2}F_{1}$-identities
 include the famous \emph{Gauss summation theorem} \cite[\S1.1]{Slater1966} 
\begin{equation}\label{displayGaussfirst}
 {}_{2}F_{1}\!\!\left[ \begin{matrix} 
 a, b \vspace{1mm}\\ 
 c \end{matrix} \ \Bigg| \ 1 
 \right] = \frac{\Gamma(c)\,\Gamma(c-a-b)}
{\Gamma(c-a)\,\Gamma(c-b)}, 
\end{equation}
 along with 
 \emph{Gauss's second summation theorem} \cite[\S1.7]{Slater1966} 
\begin{equation}\label{Gausssecond}
 {}_{2}F_{1}\!\!\left[ \begin{matrix} 
 a, b \vspace{1mm}\\ 
 \frac{a+b+1}{2} \end{matrix} \ \Bigg| \ \frac{1}{2} 
 \right]
 = \frac{ \Gamma\left( \frac{1}{2} \right) \Gamma\left( \frac{a+b+1}{2} \right) }{ \Gamma\left( 
 \frac{a+1}{2} \right) \Gamma\left( \frac{b+1}{2} \right) }, 
\end{equation}
 together with \emph{Bailey's summation theorem} \cite[\S1.7]{Slater1966} 
\begin{equation}\label{Baileytheorem}
 {}_{2}F_{1}\!\!\left[ \begin{matrix} 
 a, 1-a \vspace{1mm}\\ 
 c \end{matrix} \ \Bigg| \ \frac{1}{2} 
 \right]
 = \frac{\Gamma\left( \frac{c}{2} \right) \Gamma\left( \frac{c+1}{2} \right) }{ 
 \Gamma\left( \frac{c+a}{2} \right) \Gamma\left( \frac{c-a+1}{2} \right) } 
\end{equation}
 and \emph{Kummer's theorem} \cite[\S1.7]{Slater1966} 
\begin{equation}\label{Kummertheorem}
 {}_{2}F_{1}\!\!\left[ \begin{matrix} 
 a, b \vspace{1mm}\\ 
 1 + a - b \end{matrix} \ \Bigg| \ -1 
 \right]
 = 2^{-a} \frac{ \Gamma(1+a-b) \Gamma\left( 1 + \frac{a}{2} \right) }{ \Gamma(1+a) \Gamma\left( 
 1 + \frac{a}{2} - b \right) }. 
\end{equation}
 Often, elementary and special functions can be expressed by specifying rational values for the parameters in $_2F_1(z)$-series, as in 
 the Maclaurin series expansion 
\begin{equation}\label{expandarcsin}
 {}_{2}F_{1}\!\!\left[ \begin{matrix} 
 \frac{1}{2}, \frac{1}{2} \vspace{1mm}\\ 
 \frac{3}{2} \end{matrix} \ \Bigg| \ z 
 \right]
 = \frac{\arcsin\left( \sqrt{z} \right) }{\sqrt{z}} 
\end{equation}
 or the Maclaurin series expansion 
\begin{equation}\label{expandK} 
 {}_{2}F_{1}\!\!\left[ \begin{matrix} 
 \frac{1}{2}, \frac{1}{2} \vspace{1mm}\\ 
 1 \end{matrix} \ \Bigg| \ z 
 \right]
 = 
 \frac{2}{\pi} \text{{\bf K}}\left( \sqrt{z} \right) 
\end{equation}
 associated with the special function 
\begin{equation}\label{displayK}
 \text{{\bf K}}(k) = \int_{0}^{\frac{\pi}{2}} \left( 1 - k^2 \sin^2 \theta \right)^{-1/2} \, d\theta 
\end{equation}
 referred to as the complete elliptic integral of the first kind. 
 
 In view of the arguments $z \in \{ \pm1, \frac{1}{2} \}$ among the classical hypergeometric identities among \eqref{displayGaussfirst}, 
 \eqref{Gausssecond}, \eqref{Baileytheorem}, and \eqref{Kummertheorem}, and in view of the parameters arising among 
 $_{p}F_{q}$-expansions for elementary and special function as in \eqref{expandarcsin} and \eqref{expandK}, it should be emphasized 
 that it is only in very exceptional cases that given $_2F_1$-series with a rational argument and with rational parameters can be explicitly 
 expressed, i.e., as a combination of $\Gamma$-values with rational arguments. This emphasizes the remarkable nature of the moviating 
 result highlighed in Section \ref{subsectionmotive} below, and this may be further illustrated with the special values for 
 ${}_{2}F_{1}$-series due to Joyce and Zucker \cite{JoyceZucker1991,ZuckerJoyce2001,JoyceZucker2002} and reviewed 
 in Section \ref{secBack}. 

\subsection{A motivating result}\label{subsectionmotive}
 A main result introduced in this paper is the evaluation 
\begin{equation}\label{maineval}
 {}_{2}F_{1}\!\!\left[ \begin{matrix} 
 \frac{7}{48}, \frac{31}{48} \vspace{1mm}\\ 
 \frac{9}{8} \end{matrix} \ \Bigg| \ \left(\frac{172872}{185039}\right)^2 
 \right] = \frac{185039^{7/24}\,\Gamma^3\!\left(\tfrac{1}{8}\right)\Gamma\!\left(\tfrac{5}{8}\right)}
{672\,(1+\sqrt{2})\,3^{1/8}\,\pi^2}. 
\end{equation}

 Ebisu \cite{Ebisu2017} derives many so-called ``strange'' ${}_{2}F_{1}$-identities for constants given by combinations of $\Gamma$-values 
 with rational arguments, using an innovative applications of Appell series. Informally, the degree of strangeness can be thought of as 
 being given by the size of the denominator of the ration argument of a given ${}_{2}F_{1}$-series that is not evaluable in a direct way via 
 classically known hypergeometric identities, with larger values denoting greater ``strangeness.'' Our new result highlighed in 
 \eqref{maineval} appears to have the greatest degree of ``strangeness,'' out of known ${}_{2}F_{1}$-evaluations. 

\section{Background}\label{secBack}
 Evaluations for $_2F_1$-series as in \eqref{maineval} and as in the Joyce--Zucker evaluations 
 \cite{JoyceZucker1991,ZuckerJoyce2001,JoyceZucker2002} reviewed in Section \ref{subsecJZ} below 
 may be seen as appearing sporadically and as resisting systematic derivations. One may compare this to the work of 
 Gessel and Stanton \cite{GesselStanton1982}, who
 identified what may be seen as 
 genuinely non-classical $\Gamma$-product evaluations for so-called ``strange'' hypergeometric series. 

 Computational approaches due to Apagodu and Zeilberger \cite{ApagoduZeilberger2008} and obtained from Zeilberger's 
 computer system {\tt EKHAD}\footnote{See 
 \url{https://sites.math.rutgers.edu/~zeilberg/mamarim/mamarimhtml/strange.html}.}
 yield further families of so-called ``strange'' 
 ${}_{2}F_1$-identities, such as 
\begin{equation*}
 {}_{2}F_{1}\!\!\left[ \begin{matrix} 
 -n, -n - \frac{1}{2} \vspace{1mm}\\ 
 4n + \frac{9}{2} \end{matrix} \ \Bigg| \ \frac{1}{5} 
 \right]
 = \left( \frac{16384}{15625} \right)^{n} 
 \left[ \begin{matrix} \frac{9}{8}, \frac{11}{8}, \frac{13}{8}, \frac{15}{8} \vspace{1mm} \\ 
 \frac{6}{5}, \frac{9}{5}, \frac{13}{10}, \frac{17}{10} \end{matrix} \right]_{n} 
\end{equation*}
 Since our main technique relies on transforming known and non-classical ${}_{2}F_{1}(r)$-identities for a fixed and rational convergence 
 rate $r \not\in \{ \pm 1, \frac{1}{2} \}$ to produce an evaluation for a new hypergeometric series of a different convergence rate, 
 this may be seen in relation to the work of Campbell \cite{Campbell2023Guillera}, who showed
 how Zeilberger's ${}_{2}F_{1}$-identities 
 can be transformed, via the Wilf--Zeilberger method \cite{PetkovsekWilfZeilberger1996}, to produce closed-form 
 evaluations for accelerated hypergeometric series. 

 Further modern developments include numerous research works concerning what is referred to as \emph{Gosper's strange series} 
 \cite{Campbell2023Gosper,Chu2017,Ebisu2013}, namely $$ {}_{2}F_{1}\!\!\left[ \begin{matrix} 1-a, b \vspace{1mm} \\ b + 
 2 \end{matrix} \ \Bigg| \ \frac{b}{a+b} \right] = (b+1) \left( \frac{a}{a+b} \right)^{a}. $$ Moreover, research from Ebisu \cite{Ebisu2017} 
 suggests a deeper arithmetic structure underlying the above phenomenon outlined above concerning so-called ``strange'' 
 ${}_{2}F_{1}$-series. This may be seen in relation to the wide variety of methods in the derivation of closed forms for such series, as in 
 the work of Campbell and Levrie \cite{CampbellLevrie2024On}, who applied Zeilberger's algorithm 
 \cite[\S6]{PetkovsekWilfZeilberger1996} to prove the relation 
\begin{equation}\label{CLclosed}
 {}_{2}F_{1}\!\!\left[ \begin{matrix} 
 \frac{1}{2}, \frac{2}{3} \vspace{1mm}\\ 
 \frac{1}{6} \end{matrix} \ \Bigg| \ \frac{1}{4} 
 \right] =
 \frac{4}{3} \sqrt[3]{2}, 
\end{equation}
 directly inspired by the Joyce--Zucker series reviewed in Section \ref{subsecJZ}. 
 
 Restricting the convergence rate of ${}_{2}F_{1}(z)$ drastically reduces the number of valid identities, making those that remain 
 appear especially elegant, as illustrated with 
 our new result in \eqref{maineval}, 
 along with known closed forms as in \eqref{CLclosed}. 
 This motivates our focus on rational arguments as a natural setting for studying hypergeometric ``strangeness.''

\subsection{Joyce and Zucker's special values of hypergeometric series}\label{subsecJZ}
 As highlighted in the Wolfram Mathworld site for hypergeometric 
 functions\footnote{See \url{https://mathworld.wolfram.com/HypergeometricFunction.html}}, Zucker and Joyce \cite{ZuckerJoyce2001}
 introduced remarkable and algebraic evaluations for ${}_{2}F_{1}$-series with 
 rational arguments and parameters, including 
\begin{align*}
 {}_{2}F_{1}\!\!\left[ \begin{matrix} 
 \frac{1}{8}, \frac{3}{8} \vspace{1mm}\\ 
 \frac{1}{2} \end{matrix} \ \Bigg| \ \frac{2400}{2401} 
 \right] & =
 \frac{2}{3} \sqrt{7}, \\ 
 {}_{2}F_{1}\!\!\left[ \begin{matrix} 
 \frac{1}{6}, \frac{1}{3} \vspace{1mm}\\ 
 \frac{1}{2} \end{matrix} \ \Bigg| \ \frac{25}{27}
 \right] & =
 \frac{3}{4} \sqrt{3}, \\ 
 {}_{2}F_{1}\!\!\left[ \begin{matrix} 
 \frac{1}{6}, \frac{1}{2} \vspace{1mm}\\ 
 \frac{2}{3} \end{matrix} \ \Bigg| \ \frac{125}{128}
 \right] & =
 \frac{4}{3} \sqrt[6]{2}, \ \text{and} \\ 
 {}_{2}F_{1}\!\!\left[ \begin{matrix} 
 \frac{1}{12}, \frac{5}{12} \vspace{1mm}\\ 
 \frac{1}{2} \end{matrix} \ \Bigg| \ \frac{1323}{1331} \right] & =
 \frac{3}{4} \sqrt[4]{11}. 
\end{align*} 

 The crux of Zucker and Joyce's derivations of the algebraic closed forms listed above is given by a combination of applications of the 
 classical known ${}_2F_{1}$-transform \cite[\S2]{ErdelyiMagnusOberhettingerTricomi1953}
\begin{multline*}
 \frac{2 \Gamma\left( \frac{1}{2} \right) \Gamma\left( a + b + \frac{1}{2} 
 \right) }{ \Gamma\left( a + \frac{1}{2} \right) \Gamma\left( b + \frac{1}{2} \right) } 
 {}_{2}F_{1}\!\!\left[ \begin{matrix} 
 a, b \vspace{1mm}\\ 
 \frac{1}{2} \end{matrix} \ \Bigg| \ z 
 \right] = \\ 
 {}_{2}F_{1}\!\!\left[ \begin{matrix} 
 2a, 2b \vspace{1mm}\\ 
 a + b + \frac{1}{2} \end{matrix} \ \Bigg| \ \frac{1-\sqrt{z}}{2} 
 \right] + 
 {}_{2}F_{1}\!\!\left[ \begin{matrix} 
 2a, 2b \vspace{1mm}\\ 
 a + b + \frac{1}{2} \end{matrix} \ \Bigg| \ \frac{1 + \sqrt{z}}{2} \right] 
\end{multline*} 
 together with application of \emph{elliptic integral singular values}, i.e., special values (given by combinations of $\Gamma$-values
 with rational arguments) for the function in \eqref{displayK} at algebraic arguments. 
 This is in contrast to our approach, which does not involve the theory of elliptic integrals. 

\section{Main results}
 Of central interest, for the purposes of this paper, is the hypergeometric identity 
 \begin{equation}\label{mainGosper}
 {}_{2}F_{1}\!\!\left[ \begin{matrix} 
 \frac{1}{2}, b \vspace{1mm}\\ 
 \frac{5}{2} - 2 b \end{matrix} \ \Bigg| \ \frac{1}{4} 
 \right] =
 \frac{2^{2b} \sqrt{\pi}}{3} \frac{ \Gamma\left( \frac{5}{2} - 2 b \right) }{ \Gamma^{2}\left( 
 \frac{3}{2} - b \right)} 
\end{equation}
 attributed to Gosper, as in the work of Vidunas \cite{Vidunas2002}, who introduced a generalization of \eqref{mainGosper}, and who 
 sketched a difference equations-based way of deriving this generalization. While \eqref{mainGosper} can be derived through an 
 application of Zeilberger's algorithm, we introduce a proof of \eqref{mainGosper} that is motivated by how the methods involved in 
 this proof may be seen as laying a foundation for the derivation of new results as in Section \ref{subsectionmotive}. To the best of our 
 knowledge, no full or explicit proof of \eqref{mainGosper} has previously been given, despite how \eqref{mainGosper} forms a 
 cornerstone for what are described as \textit{strange} ${}_{2}F_{1}$-identities, motivating our new and integration-based proof below. 
 The research interest in our proof of Gosper's ${}_{2}F_{1}\left( \frac{1}{4} \right)$-relation in \eqref{mainGosper} may also be seen in 
 relation to the recent work of Chen and Chu on the extension of ${}_{3}F_{2}\left( \frac{3}{4} \right)$-identities due to Gosper 
 \cite{ChenChu2025}, and (with regard to the convergence rate of $\frac{1}{4}$ in \eqref{mainGosper}) in relation to a number of recent 
 research works concerning fast-congerving hypergeometric series due to
 Gosper \cite{Campbelltoappear,CampbellLevrie2024Proof,Chu2022,Nimbran2024}. 

 As in the below proof, we make use of the beta integral
 $$ \operatorname{B}(z_1, z_2) = \int_{0}^{1} t^{z_1-1} (1-t)^{z_2-1} \, dt $$
 together with the relation 
\begin{equation}\label{BGamma}
 \operatorname{B}(z_1,z_2) = \frac{\Gamma(z_1) \Gamma(z_2)}{\Gamma(z_1 + z_2)}.
\end{equation} 

\begin{theorem}\label{theoremGosper}
 Gosper's formula in \eqref{mainGosper} holds for all values $b$ such that $b$ such that the left-hand side of \eqref{mainGosper} 
 converges and such that its lower parameter is not a an integer in $\mathbb{Z}_{\leq 0}$ (cf.\ \cite{Vidunas2002}). 
\end{theorem}

\begin{proof}
 We begin by applying Euler's ${}_{2}F_{1}$-transform such that  $$ {}_{2}F_{1}\!\!\left[ \begin{matrix}   a, b \vspace{1mm}\\   c \end{matrix} 
 \ \Bigg| \ z   \right] = (1-z)^{c-a-b} {}_{2}F_{1}\!\!\left[ \begin{matrix} 
 c - a, c - b \vspace{1mm}\\ 
 c \end{matrix} \ \Bigg| \ z
 \right] $$
 Applying this with $a=\frac{1}{2}$, $c=\frac{5}{2}-2b$, and $z=\frac{1}{4}$ 
 and then using Euler's integral relation in \eqref{Eulerint}, we obtain 
\begin{multline*}
 {}_{2}F_{1}\!\!\left[ \begin{matrix} 
 2-2b, \frac{5}{2}-3b \vspace{1mm}\\ 
 \frac{5}{2}-2b \end{matrix} \ \Bigg| \ \frac{1}{4} 
 \right]  = 
\frac{\Gamma\!\left(\frac{5}{2}-2b\right)}
{\Gamma\!\left(\frac{5}{2}-3b\right)\Gamma(b)} \\
 \times
\int_0^1
t^{\frac{3}{2}-3b}
(1-t)^{b-1}
\left(1-\frac{t}{4}\right)^{2b-2} \,dt.
\end{multline*}
 Enforcing the substitution $t = \frac{4u}{(1+u)^2}$ yields 
\begin{equation*}
4^{\frac{5}{2}-3b}
\int_0^1
u^{\frac{3}{2}-3b}
(1-u)^{2b-1}
(1+u+u^2)^{2b-2}
\,du.
\end{equation*}
 and, by exploiting the cyclotomic relation $ (1-u)(1+u+u^2) = 1-u^3$, we obtain
\begin{equation*}
\int_0^1 \left( u^{\frac{3}{2}-3b} - u^{\frac{5}{2}-3b} \right) (1-u^3)^{2b-2} \, du.
\end{equation*}
 After enforcing a substitution of the form $v = u^3$, we find that $$ {}_{2}F_{1}\!\!\left[ \begin{matrix} 
 2 - 2b, \frac{5}{2} - 3b \vspace{1mm}\\ 
 \frac{5}{2} - 2b \end{matrix} \ \Bigg| \ \frac{1}{4} 
 \right] = \frac{\Gamma\!\left(\frac{5}{2} - 2b\right)}
{3\,\Gamma\!\left(\frac{5}{2} - 3b\right)\Gamma(b)}
\,4^{\frac{5}{2} - 3b}
\Delta, $$
 writing
\begin{equation*}
 \Delta = \mathrm{B}\left(\frac{5}{6} - b, 2b - 1\right) - \mathrm{B}\left(\frac{7}{6} - b, 2b - 1\right). 
\end{equation*} 
 From the relation in \eqref{BGamma} together with the reflection formula for the $\Gamma$-function, we obtain 
\begin{equation*}
 \Delta = -\frac{\Gamma(2b - 1) \Gamma\left(\frac{5}{6} - b\right) \Gamma\left(\frac{7}{6} - b\right) \cos(\pi b)}{\pi}. 
\end{equation*}
 Routine applications of the Gauss multiplication formula and the above reflection formula 
 give an equivalent formulation of the desired result. 
 \end{proof}

\subsection{A new $_2F_1$-evaluation}
 We proceed, using Theorem \ref{theoremGosper}, to construct a proof of the motivating result from Section \ref{subsectionmotive}. 
 
 Applying the known quadratic transform 
 \begin{equation*}
 {}_{2}F_{1}\!\!\left[ \begin{matrix} a, b \vspace{1mm}\\ 2 b \end{matrix} \ \Bigg| \ z \right] = (1-z)^{b-a} \left( 1 - \frac{z}{2} \right)^{a - 
 2 b} {}_{2}F_{1}\!\!\left[ \begin{matrix} b - \frac{a}{2}, b + \frac{1-a}{2} \vspace{1mm} \\ b + \frac{1}{2} \end{matrix} \ \Bigg| \ \frac{z^2}{(2 
 -z)^2} \right] 
\end{equation*}
 and setting $b = \frac{5}{8}$ and $a = \frac{1}{2}$, we write 
\begin{equation}\label{afterfirstquad}
 {}_{2}F_{1}\!\!\left[ \begin{matrix} 
 \frac{1}{2}, \frac{5}{8} \vspace{1mm}\\ 
 \frac{5}{4} \end{matrix} \ \Bigg| \ z 
 \right] = 
 (1-z)^{1/8}\left(1-\frac{z}{2}\right)^{-3/4} {}_{2}F_{1}\!\!\left[ \begin{matrix} 
 \frac{3}{8}, \frac{7}{8} \vspace{1mm}\\ 
 \frac{9}{8} \end{matrix} \ \Bigg| \ \frac{z^2}{(2-z)^2}
 \right]. 
\end{equation}
 To the right-hand side of \eqref{afterfirstquad}, we apply the known quadratic transform 
 \begin{equation*}
 {}_{2}F_{1}\!\!\left[ \begin{matrix} 
 a, b \vspace{1mm}\\ 
 \frac{a+b+1}{2} \end{matrix} \ \Bigg| \ z 
 \right] = (1-2z)^{-a} {}_{2}F_{1}\!\!\left[ \begin{matrix} 
 \frac{a}{2}, \frac{a+1}{2} \vspace{1mm}\\ 
 \frac{a+b+1}{2} \end{matrix} \ \Bigg| \ \frac{4z(z-1)}{(2z-1)^2} 
 \right], 
\end{equation*}
 i.e., with $a = \frac{7}{8}$ and $b = \frac{3}{8}$, i.e., so that the right-hand side of \eqref{afterfirstquad}
 may be expressed as 
\begin{equation*}
 (1 - z)^{1/8} \left(1 - \frac{z}{2}\right)^{-3/4} \frac{(2 - z)^{7/4}}{(4 - 4z - z^2)^{7/8}} \, {}_{2}F_{1}\!\!\left[ \begin{matrix} 
 \frac{7}{16}, \frac{15}{16} \vspace{1mm}\\ 
 \frac{9}{8} \end{matrix} \ \Bigg| \ \frac{16z^2(z - 1)}{(4 - 4z - z^2)^2} 
 \right]. 
\end{equation*}
 Applying the cubic transformation: $$ {}_{2}F_{1}\!\!\left[ \begin{matrix} a, a + \frac{1}{2} \vspace{1mm}\\ \frac{4a + 5}{6} \end{matrix} 
 \ \Bigg| \ z \right] 
 = \frac{1}{(1 - 9z)^{2a/3}} {}_{2}F_{1}\!\!\left[ \begin{matrix} 
 \frac{a}{3}, \frac{a}{3} + \frac{1}{2} \vspace{1mm}\\ 
 \frac{4a + 5}{6} \end{matrix} \ \Bigg| \ -\frac{27z(1 - z)^2}{(1 - 9z)^2} 
 \right] $$
 and setting $a = \frac{7}{16}$ and repeating the above procedure, we obtain a $12^{\text{th}}$-degree ${}_{2}F_{1}$-transform 
 involving the argument $$-\frac{432 (z - 2)^8 (z - 1) z^2}{\left(z^2 + 4z - 4\right)^2 \left(z^4 - 136z^3 + 152z^2 - 32z + 16\right)^2}.$$ 
 After setting 
 $z=1/4$, we obtain $\left(\frac{172872}{185039} \right)^2$, and, by using 
 Theorem \ref{theoremGosper} with $b = \frac{5}{8}$, 
 we obtain the desired evaluation in Theorem \ref{theoremGosper}. 

\section{Conclusion}
 Using something of a variant our proof of Theorem \ref{theoremGosper}, by applying the Euler relation in \eqref{Eulerint} in conjunction 
 with the closed form in \eqref{CLclosed}, we have found that 
\begin{equation*}
 {}_{2}F_{1}\!\!\left[ \begin{matrix} 
 \frac{1}{2}, \frac{3}{2} \vspace{1mm}\\ 
 \frac{13}{6} \end{matrix} \ \Bigg| \ -\frac{1}{3} 
 \right] =
 \frac{7}{2^{2/3} \sqrt{3}} - \frac{7 \Gamma^{3}\left( \frac{1}{6} \right)}{2^{14/3} 3^{3/2} \pi^{3/2}}, 
\end{equation*}
 and this appears to be new. This motivates a full exploration as to how our integration-based techniques above could be extended, and 
 we encourage a full exploration of this. 

\subsection*{Acknowledgements}
 The author thanks Dr.\ John M.\ Campbell for useful feedback and for help formatting and organizing this paper.

\end{document}